\documentclass[11pt,reqno]{amsart} 
\usepackage[margin=1cm]{geometry}
\usepackage{lipsum} 
\usepackage{amsmath, amssymb, amsthm, mathabx,graphicx}
\usepackage{mathrsfs, color}
\usepackage{url}
\usepackage{stackrel} 
\usepackage{courier}
\usepackage[all,cmtip]{xy} 
\usepackage{multicol,arydshln}
\usepackage{fullpage}
\usepackage{cancel,comment} 
\usepackage[usenames,dvipsnames]{xcolor}
\usepackage{parcolumns}

\usepackage[colorlinks=true, pdfstartview=FitV, linkcolor=blue, citecolor=blue, urlcolor=blue]{hyperref}	

\def\CircleArrowleft{\ensuremath{%
  \reflectbox{\rotatebox[origin=c]{180}{$\circlearrowleft$}}}}

\definecolor{darkred}{rgb}{0.7,0,0} 

\definecolor{darkgreen}{rgb}{0,0.6,0} 
\definecolor{darkblue}{rgb}{0,0,0.6} 

\catcode`~=11 
\newcommand{\urltilde}{\kern -.15em\lower .7ex\hbox{~}\kern .04em}  
\catcode`~=13 

\newcommand{\CC}{\mathbb{C}}

\newcommand{\fb}{\mathfrak{b}}

\newcommand{\diag}{\mathop{\mathrm{diag}}\nolimits}

\newcommand{\Proj}{\mathop{\mathrm{Proj}}\nolimits}

\newcommand{\sm}{\mathop{\mathrm{sm}}\nolimits}
\newcommand{\Ad}{\mathop{\mathrm{Ad}}\nolimits}

\newcommand{\Spec}{\mathop{\mathrm{Spec}}\nolimits}

\newcommand{\I}{\mathop{\mathbf{I}}\nolimits}

\newcommand{\sing}{\mathop{\mathrm{sing}}\nolimits}

\newcommand{\tr}{\mathop{\mathrm{tr}}\nolimits}

\newcommand{\Lie}{\mathop{\mathrm{Lie}}\nolimits}

\newcommand{\rss}{\mathop{\mathrm{rss}}\nolimits}

\makeatletter
\newif\if@borderstar
   \def\bordermatrix{\@ifnextchar*{%
       \@borderstartrue\@bordermatrix@i}{\@borderstarfalse\@bordermatrix@i*}%
   }
   \def\@bordermatrix@i*{\@ifnextchar[{\@bordermatrix@ii}{\@bordermatrix@ii[()]}}
   \def\@bordermatrix@ii[#1]#2{%
   \begingroup
     \m@th\@tempdima8.75\p@\setbox\z@\vbox{%
       \def\cr{\crcr\noalign{\kern 2\p@\global\let\cr\endline }}%
       \ialign {$##$\hfil\kern 2\p@\kern\@tempdima & \thinspace %
       \hfil $##$\hfil && \quad\hfil $##$\hfil\crcr\omit\strut %
       \hfil\crcr\noalign{\kern -\baselineskip}#2\crcr\omit %
       \strut\cr}}%
     \setbox\tw@\vbox{\unvcopy\z@\global\setbox\@ne\lastbox}%
     \setbox\tw@\hbox{\unhbox\@ne\unskip\global\setbox\@ne\lastbox}%
     \setbox\tw@\hbox{%
       $\kern\wd\@ne\kern -\@tempdima\left\@firstoftwo#1%
         \if@borderstar\kern2pt\else\kern -\wd\@ne\fi%
       \global\setbox\@ne\vbox{\box\@ne\if@borderstar\else\kern 2\p@\fi}%
       \vcenter{\if@borderstar\else\kern -\ht\@ne\fi%
         \unvbox\z@\kern-\if@borderstar2\fi\baselineskip}%
         \if@borderstar\kern-2\@tempdima\kern2\p@\else\,\fi\right\@secondoftwo#1 $%
     }\null \;\vbox{\kern\ht\@ne\box\tw@}%
   \endgroup
   }
\makeatother

\theoremstyle{plain}
\newtheorem{theorem}{Theorem}[section] 
\newtheorem{proposition}[theorem]{Proposition}
\newtheorem{lemma}[theorem]{Lemma}
\newtheorem{conjecture}[theorem]{Conjecture} 
\newtheorem{corollary}[theorem]{Corollary}

\theoremstyle{remark}
\newtheorem{definition}[theorem]{Definition}
\newtheorem{remark}[theorem]{Remark}

\usepackage[colorinlistoftodos]{todonotes}

\begin{document} 

\title[Towards the affine and GIT quotients of the Borel moment map]{Towards the affine and geometric invariant theory quotients of the Borel moment map} 

\author{Mee Seong Im}
\address{Department of Mathematical Sciences, United States Military Academy, West Point, NY 10996 USA}
\email{meeseongim@gmail.com}

\author{Meral Tosun}
\address{Department of Mathematics, Galatasary University, Ortak\"oy 34357, Istanbul, Turkey}
\email{mrltosun@gmail.com}


\date{\today}

\begin{abstract}
We study the Borel moment map $\mu_B:T^*(\mathfrak{b}\times \mathbb{C}^n)\rightarrow \mathfrak{b}^*$, given by $(r,s,i,j)\mapsto [r,s]+ij$, and describe our algorithm to construct the geometric invariant theory (GIT) quotients 
$\mu_B^{-1}(0)/\!\!/_{\det}B$ and $\mu_B^{-1}(0)/\!\!/_{\det^{-1}}B$, 
and the affine quotient $\mu_B^{-1}(0)/\!\!/B$. 
We also provide an insight of the singular locus of $2^n$ irreducible components of $\mu_B$. 
Finally, analogous to the Hilbert--Chow morphism,  
we discuss that the GIT quotient for the Borel setting is a resolution of singularities. 
\end{abstract}  

\maketitle 



\section{Introduction}\label{section:intro} 

Parabolic equivariant geometry frequently appears in algebraic geometry, representation theory, and mathematical physics. They generalize equivariant morphisms for reductive groups, with multitude of connections to quiver Hecke algebras (cf.~\cite{MR2525917,MR2763732,MR2908731,rouquier2008}), isospectral Hilbert schemes (cf.~\cite{MR1839919,MR3836769}), flag Hilbert schemes (cf.~\cite{GNR16}), and partial  (Grothendieck--)Springer resolutions (cf.~\cite{MR2838836,Nevins-GSresolutions,MR3836769,MR3312842}).

We will work over the set of complex numbers. 
Let $G=GL_n(\mathbb{C})$, the set of $n\times n$ invertible matrices over $\mathbb{C}$, and let $P$ be a parabolic group of $G$ consisting of invertible block upper triangular matrices. Let $B$ be the set of invertible upper triangular matrices in $G$. Note that $P\supseteq B$.   
Let $\mathfrak{g}= \Lie(G)$, $\mathfrak{p}=\Lie(P)$, and $\mathfrak{b}=\Lie(B)$. 

Consider the matrix variety $\mathfrak{b}\times \mathbb{C}^n$, and its cotangent bundle $T^*(\mathfrak{b}\times \mathbb{C}^n)$, where we make the following identification: 
\[ 
T^*(\mathfrak{b}\times \mathbb{C}^n) = \mathfrak{b}\times \mathfrak{b}^*\times \mathbb{C}^n \times (\mathbb{C}^n)^*, 
\] 
where $\mathfrak{b}^*\stackrel{\tr}{\cong}\mathfrak{g}/\mathfrak{u}^+$ and $\mathfrak{u}^+$ is the nilpotent radical in $\mathfrak{b}$. 
There is a natural $B$-action on $\mathfrak{b}\times \mathbb{C}^n$ via $b.(r,i)=(brb^{-1}, bi)$. Taking the derivative of this action gives us 
\[ 
a:\mathfrak{b}\rightarrow \Gamma(T_{\mathfrak{b}\times \mathbb{C}^n}) \subseteq \mathbb{C}[T^*(\mathfrak{b}\times \mathbb{C}^n)], \quad 
\mbox{ given by } 
\quad 
a(v)(r,i) = \frac{d}{dt}\left(g_t.(r,i)\right)\Big|_{t=0} = ([v,r],vi), 
\] 
where $g_t = \exp(tv)$. We dualize $a$ to obtain the moment map 
\begin{equation}
\label{eqn:Borel-moment-map} 
\mu_B= a^*: T^*(\mathfrak{b}\times \mathbb{C}^n)\rightarrow \mathfrak{b}^*, 
\quad 
\mbox{ where } 
\quad  
(r,s,i,j)\mapsto [r,s]+ ij. 
\end{equation} 
Note that the $B$-action is induced onto the cotangent bundle, giving us 
\[ 
B\:\CircleArrowleft\: T^*(\mathfrak{b}\times \mathbb{C}^n) 
\quad 
\mbox{ via }
\quad  
b.(r,s,i,j) = (brb^{-1}, bsb^{-1}, bi, jb^{-1}). 
\]

In \cite{MR3836769}, the first author restricts to the regular semisimple locus $\mu_B^{-1}(0)^{\rss}$ of the components of $\mu_B$, i.e., $r$ has pairwise distinct eigenvalues, and studies its affine quotient $\mu_B^{-1}(0)^{\rss}/\!\!/B$. The author shows that the affine quotient is isomorphic to $\mathbb{C}^{2n}\setminus \Delta$, where $\Delta=\{(a_1,\ldots,a_n,0,\ldots, 0): 
a_{\iota} = a_{\gamma} \mbox{ for some }\iota\not=\gamma\}$ (cf. \cite[Thm.~1.6]{MR3836769}). 
So this locus $\mu_B^{-1}(0)^{\rss}/\!\!/B$ is contained in the smooth locus of the affine quotient, which we will denote by $X^{\sm}$. 
In this construction, $B$-invariant polynomials are explicitly given for the affine quotient.

We now recall a description of the irreducible components of $\mu_B$ (see \cite[Prop.~4.2]{Nevins-GSresolutions}). 
But first, we state \cite[Lemma~4.1]{Nevins-GSresolutions}:
\begin{lemma}[Nevins, Lemma 4.1] 
Suppose $r$ is a diagonal $n\times n$ matrix with pairwise distinct eigenvalues. Let $i\in \mathbb{C}^n$ and $j\in (\mathbb{C}^n)^*$. Then given diagonal entries $s_{aa}$ for an $n\times n$ matrix $s$, there is a unique $s$ satisfying $[r,s]+ij=0$. In particular, if $\overline{s}\in \mathfrak{g}/\mathfrak{u}^+$ and $\mu_B(r,\overline{s},i,j)=0$, then there is a unique lift of $\overline{s}$ to $s\in\mathfrak{g}$ such that $\mu_G(r,s,i,j)=0$, where $\mu_G:T^*(\mathfrak{g}\times \mathbb{C}^n)\rightarrow \mathfrak{g}^*\stackrel{\tr}{\cong} \mathfrak{g}$.  
\end{lemma}

\begin{proposition}[Nevins, Proposition 4.2] 
\label{prop:Nevins-Prop-4.2}
For $n\leq 5$, the irreducible components of $\mu_B^{-1}(0)$ are the closures of subsets $\mathcal{C}_{\ell}$, where $\ell:\{ 1,\ldots, n\}\rightarrow \{ 0,1\}$ is a function. The subset $\mathcal{C}_{\ell}$ consists of the orbits of quadruples $(r,\overline{s},i,j)$, where $r$ is diagonal and has pairwise distinct eigenvalues, $i_k=\ell(k)$, $j_k=1-\ell(k)$, and $\overline{s}$ is the image of a matrix $s\in \mathfrak{g}$ in $\mathfrak{g}/\mathfrak{u}^+$ that has arbitrary diagonal entries and off-diagonal entries 
\[ 
s_{ab} = -\frac{(ij)_{ab}}{r_{aa} - r_{bb}}
\] 
for $a\not=b$. 
\end{proposition}
Thus $\mathcal{C}_{\ell}$, as discussed in Proposition~\ref{prop:Nevins-Prop-4.2}, enumerate $2^n$ irreducible components of $\mu_B^{-1}(0)$. 
We will write $\overline{\mathcal{C}}_0$ to be the closure of the irreducible component enumerated by $\ell(k)=0$ for all $1\leq k\leq n$, i.e., $i=0\in \mathbb{C}^n$ while $j=(1,\ldots, 1)\in (\mathbb{C}^n)^*$. 
 
Note that the first author has verified using \texttt{Macaulay2} (cf.~\cite{GS}) and computational algebraic geometry techniques that $\mu_B$ is a complete intersection for up to $n=5$.

Furthermore, Im--Scrimshaw in \cite{Im-Scrimshaw-parabolic} prove that given a parabolic subalgebra with at most $5$ Jordan blocks, the components of $\mu_P$ form a complete intersection, where $\mu_P:T^*(\mathfrak{p}\times \mathbb{C}^n)\rightarrow \mathfrak{p}^*$, which is a generalization of the Borel moment map $\mu_B$ given in \eqref{eqn:Borel-moment-map}. The irreducible components are enumerated, where the vector $i$ and the covector $j$ are $\{0,1\}$-vectors (cf. \cite[Thm.~1.1]{Im-Scrimshaw-parabolic}), which are similar to Proposition~\ref{prop:Nevins-Prop-4.2}, and are equidimensional (cf.~\cite[Thm.~1.2]{Im-Scrimshaw-parabolic}).

In this manuscript, we specialize when $P=B$ and provide our program to describe the entire affine quotient $\mu_B^{-1}(0)/\!\!/B$, and the geometric invariant theory (GIT) quotients $\mu_B^{-1}(0)/\!\!/_{\det}B$ and $\mu_B^{-1}(0)/\!\!/_{\det^{-1}}B$ (see \S\ref{subsection:GIT} for the definition of affine and GIT quotients). 
Although we have explicitly constructed $B$-invariant polynomials and $B$-semi-invariant polynomials, it is difficult to show that they generate the $B$-invariant subalgebra and the semi-invariant module, respectively. 

Moreover, although the regular semisimple results in \cite{MR3836769} hold for all $n$, the parabolic subalgebra results in \cite{Im-Scrimshaw-parabolic} hold for less than or equal to $5$ Jordan blocks since the authors are not aware of appropriate methods to tackle the case when the $P$-adjoint action on $\mathfrak{p}$ has infinitely-many orbits (cf.~proof of \cite[Prop.~4.2]{Im-Scrimshaw-parabolic}).  One of the key assumptions in \cite{MR3836769} is that we restrict to the locus where $r$ has pairwise distinct eigenvalues, which makes studying parabolic equivariant geometry (for any $n$) straightforward since $r$ is diagonalizable. However, for the entire Borel subalgebra, the geometry is no longer clear. 
Thus, we assume throughout this manuscript that $n\leq 5$.

We also describe the singular locus of the irreducible components of $\mu_B$ but since they are currently difficult to generalize for large $n$, we omit the computation.

\section*{Acknowledgement}
The first author thanks the Department of Mathematics at Galatasary University for warm hospitality in February 2019 when this project was initiated, and January 2020 when this project was completed.  
M.S.I. is partially funded by National Academy of Sciences in Washington, D.C.

\section{Parabolic equivariant geometry}
\label{section:parabolic-equiv-geom}

\subsection{Parabolic invariant functions}
\label{subsection:parabolic-equivariant-geom}
We begin with a preliminary background. 
We will write $\diag(r)=(r_{11},\ldots,r_{nn})$ to be an $n\times n$ diagonal matrix whose (ordered) coordinates along the diagonal are $r_{11},\ldots, r_{nn}$. 

\begin{lemma}
\label{lemma:B-inv}
We have $\mathbb{C}[\mathfrak{b}]^B\cong \mathbb{C}[r_{11},\ldots, r_{nn}]$. 
\end{lemma}

\begin{proof}
For $1 \leq \iota\leq n$, define a map $f_{\iota} \colon \fb \to \CC$, where $f_{\iota}(r) = r_{\iota\iota}$. 
For $b\in B$, we have $b.f_{\iota}(r) = f_{\iota}(b^{-1}rb) = r_{\iota\iota}$. 
So $\CC[f_1,\ldots, f_n] \subseteq \CC[\fb]^B$. 
Now, let $\lambda \colon \CC^* \to B$ be a $1$-parameter subgroup defined by 
\begin{equation}
\lambda(t)_{\iota\gamma} = 
\begin{cases}
t^{n-a} \qquad& \text{if } \iota=\gamma = a,  \\
\:\:\delta_{\iota\gamma} \qquad& \text{otherwise},
\end{cases}
\end{equation} 
where $\delta_{\iota\gamma}$ is the Kronecker delta.
Then 
\[ 
(\lambda(t).r)_{\iota\gamma}= (\lambda(t)r\lambda(t)^{-1})_{\iota\gamma}  
=
\begin{cases}
\:\:\:\: r_{\iota\iota} \qquad& \text{if } \iota = \gamma, \\ 
r_{\iota\gamma} t^{\gamma-\iota} \qquad& \text{if } \iota < \gamma, \\ 
\:\:\:\:\: 0 \qquad& \text{if } \iota > \gamma. \\ 
\end{cases} 
\] 
So 
\[
\lim_{t\rightarrow 0}(\lambda(t).r)_{\iota\gamma} = 
\begin{cases}
r_{\iota\iota} \qquad  &\mbox{ if } \iota=\gamma, \\ 
0 \qquad &\mbox{ if } \iota\not=\gamma. \\
\end{cases}
\] 
Since the off-diagonal entries of $r$ are zero, a $B$-invariant polynomial is independent of the off-diagonal coordinate functions. 
Thus $\mathbb{C}[\mathfrak{b}]^B \cong \mathbb{C}[\diag(r)]$. 
\end{proof}

\begin{lemma}
\label{lemma:B-dual-inv}
We have $\mathbb{C}[\mathfrak{b}^*]^B\cong \mathbb{C}[\tr(s)]$. 
\end{lemma}

\begin{proof}
Let $F$ be a polynomial in $\mathbb{C}[\tr(s)]$. 
Since $F(\tr(bsb^{-1}))
= F(\tr(sb^{-1}b)) 
= F(\tr(s))$ for any $b\in B$, $F$ is in the $B$-invariant subring $\mathbb{C}[\mathfrak{b}^*]^B$.  

Now suppose $F\in \mathbb{C}[\mathfrak{b}^*]^B$ and let $s\in\mathfrak{b}^*$. 
Then for a $1$-parameter subgroup $\lambda_1(t)$ with coordinates 
\[ \lambda_1(t)_{\iota\gamma} = \left\{ 
\begin{aligned}
t^{\iota-1} 			& \mbox{ if } \iota =\gamma, \\ 
0  \hspace{2mm} 	& \mbox{ if } \iota \not=\gamma, \\ 
\end{aligned}
\right.
\]
we have 
\[ 
(\lambda_1(t).s)_{\iota\gamma} 
    = (\lambda_1(t)s(\lambda_1(t))^{-1})_{\iota\gamma}  
		= \left\{ 
\begin{aligned} 
* 							\hspace{4mm} 		&	\mbox{ if } \iota < \gamma, \\ 
s_{\iota\iota} 	\hspace{3mm} 		& \mbox{ if } \iota = \gamma, \\ 
t^{\iota-\gamma}s_{\iota\gamma} & \mbox{ if }\iota > \gamma.  \\
\end{aligned} 
\right.
\] Taking the limit as $t\rightarrow 0$, we have 
\[ \lim_{t\rightarrow 0} (\lambda_1(t).s)_{\iota\gamma} = 
\left\{ 
\begin{aligned} 
* 				\hspace{1mm}			&	\mbox{ if } \iota < \gamma, \\ 
s_{\iota\iota} 							& \mbox{ if } \iota = \gamma, \\ 
0 				\hspace{1mm}			& \mbox{ if }	\iota > \gamma.  \\
\end{aligned}
\right. 
\]   
Since off-diagonal entries of $s$ are zero, our $B$-invariant polynomial $F$ 
is independent of the coordinates $\{ s_{\iota\gamma} \}_{\iota > \gamma}$. 
Now consider another $1$-parameter subgroup $\lambda_2(t)$, where 
\[\lambda_2(t)_{\iota\gamma} 
= \left\{ 
\begin{aligned} 
t^{\iota-1} 		& \mbox{ if } \iota \leq \gamma,  \\ 
0 \hspace{2mm}	&	\mbox{ if } \iota > \gamma. \\
\end{aligned} 
\right.
\] 
Then $(\lambda_2(t).s)_{\iota\gamma} = $ 
\begin{equation}\label{eq:lambda2s} 
= \left\{ 
\begin{aligned} 
* \hspace{20mm}&\mbox{ if }\iota < \gamma,  \\
t^{\iota-\gamma} \left( \sum_{k=\iota}^n s_{k \gamma} - \sum_{k=\iota}^n s_{k,\gamma-1}\right) 
 &\mbox{ if } \iota \geq \gamma, \\ 
\end{aligned} 
\right. 
\hspace{3mm}
= \hspace{3mm}
\left\{ 
\begin{aligned} 
* \hspace{20mm}& \mbox{ if } \iota < \gamma, \\
s_{\iota\iota}-\sum_{k=\iota}^n s_{k,\iota -1}+ \sum_{k=\iota+1}^n s_{k\iota}\hspace{4mm}				 & \mbox{ if } \iota = \gamma,  \\ 
t^{\iota-\gamma} \left( \sum_{k=\iota}^n s_{k\gamma} - \sum_{k=\iota}^n s_{k,\gamma-1} \right)  & \mbox{ if }  \iota > \gamma.  \\ 
\end{aligned} 
\right. 
\end{equation} 
Since $F(s)=F(s')$ for any $s'\in \overline{B.s}$ (the polynomial $F$ must take the same value on any orbit closure), the equality 
$\displaystyle{\lim_{t\rightarrow 0}F(\lambda_1(t).s)}
= 
\displaystyle{\lim_{t\rightarrow 0}F(\lambda_2(t).s)}$ must hold for any values of  
 $\{ s_{\alpha\beta}\}_{\alpha>\beta}$. 
So for each $1\leq \iota < n$ (starting with $\iota=1$ in ascending order),  
choose $\{ s_{k\iota}\}_{k>\iota}$
in $\displaystyle{\sum_{\iota < k\leq n} s_{k\iota}}$ such that 
\begin{equation}\label{eq:choosing-free-vars}
-\sum_{k=\iota+1}^n s_{k\iota}=s_{\iota\iota}-\sum_{k=\iota}^n s_{k,\iota-1}.  
\end{equation} 

Move the sum in (\ref{eq:choosing-free-vars}) to the left-hand side so that   
\[
\displaystyle{\sum_{\iota-1<k\leq n}s_{k,\iota-1}} 
-
\displaystyle{\sum_{\iota < k\leq n} s_{k,\iota}} = s_{\iota\iota} \quad \mbox{ for each } 1\leq \iota <n. 
\]  
This implies 
the sum of all such sum as $\iota$ varies from $1$ to $n-1$ is 
\[\begin{aligned}
\sum_{\iota=1}^{n-1} \left(\sum_{k=\iota}^n s_{k,\iota-1} -\sum_{k=\iota+1}^n s_{k \iota}\right) 
		& = \sum_{k=1}^n \not{s_{k0}} 
				+ \left(\sum_{\iota=2}^{n-1} \sum_{k=\iota}^n s_{k,\iota-1}
					  -\sum_{\iota=1}^{n-2} \sum_{ k=\iota+1 }^n s_{k\iota}  \right)
		   			-\sum_{k=n}^n s_{k,n-1} \\
		& = -\sum_{k=n}^n s_{k,n-1} = \sum_{\iota=1}^{n-1} s_{\iota\iota}. \\
\end{aligned}
\] 

By (\ref{eq:lambda2s}) and by choosing appropriate choices for $\{ s_{\alpha\beta}\}_{\alpha > \beta}$ in 
(\ref{eq:choosing-free-vars}), we have 
$(\lambda_2(t).s)_{\iota\iota}=0$ for each $1\leq \iota <n$ while  
\[
(\lambda_2(t).s)_{nn}=s_{nn}- 
\displaystyle{\sum_{n-1<k\leq n} s_{k,n-1}} = \tr(s).
\]   
This means all coordinate entries are zero except the $(n,n)$-entry, which is $\tr(s)$. 
Thus for $F$ in $\mathbb{C}[\mathfrak{b}^*]^B$, $F(s)$ must be of the form $F(s')$, where all coordinates of $s'$ are zero   
except the entry $s_{nn}'$, which equals $\tr(s)$. So $F$ is a polynomial in $\tr(s)$. 
\end{proof}

We will now generalize \cite[Defn~2.2]{MR3836769}. 
Let $[n]:=\{ 1,\ldots,n\}$.

\begin{definition}
\label{definition:ell-J}
Let $J\subseteq [n]$. 
Define 
\begin{equation}
\label{eqn:ellJ-matrix}
\ell^J := \ell^J(r)=\prod_{k\in J} (r-r_{kk}\I),   
\end{equation}
where $\I$ is the $n\times n$ identity matrix. 
Let $\tr^J :=  \tr(\ell^J)$. 
Then 
\[ 
L^J := L^J(r) = (\tr^J)^{-1} \ell^J. 
\] 
\end{definition}
In the case when $J=\{k\}$, we may write $\ell^k := \ell^k(r) = r - r_{kk}\I$, and when 
$J=\varnothing$, then $\ell^{J}:= \I$. 
Furthermore, we will write the coordinates of $\ell^J$ as $\ell^J_{\gamma\mu}$. 


\begin{lemma}
\label{lem:ellJ-B-invariant}
For any $b\in B$, $\ell^J(\Ad_b(r))=\Ad_b(\ell^J(r))$, where the adjoint action is by conjugation. 
\end{lemma}

\begin{proof}
For any $b\in B$, 
\[ 
\ell^{J}(\Ad_b(r)) =   
\prod_{
k\in J
}    
   brb^{-1} - r_{kk}\I    
= b \left(     
   \prod_{k\in J} 
   \left(  r-r_{kk}\I 
 					 \right) 
  \right) b^{-1}   \\    
   = \Ad_b(\ell^{J}(r)).      
\] 
\end{proof}

\begin{corollary}
\label{cor:ellJ-B-invariant}
For any $b\in B$, $L^J(\Ad_b(r)) = \Ad_b(L^J(r))$, where conjugation is the adjoint action. 
\end{corollary}

\begin{proof}
This follows from Lemma~\ref{lem:ellJ-B-invariant}. 
\end{proof}

\subsection{Geometric invariant theory}
\label{subsection:GIT}

We refer to \cite{MR1304906,MR2537067} for extensive background in geometric invariant theory. 

Let $X$ be a variety (or a scheme) with an action by an algebraic group $G$. Then the $G$-invariant polynomial ring is defined as 
\begin{equation}
\mathbb{C}[X]^G := \{ f\in \mathbb{C}[X] : f(g.x)=f(x) \mbox{ for all }x\in X, g\in G\}. 
\end{equation}

Let $\chi:G\rightarrow \mathbb{C}^*$ be a character of $G$, i.e., a group homomorphism. Then $\chi$-semi-invariant module is defined to be
\begin{equation}
\mathbb{C}[X]^{G,\chi}:= \{f\in \mathbb{C}[X]: f(g.x)=\chi(g)f(x) \mbox{ for all }x\in X, g\in G\}. 
\end{equation}

The affine quotient of $X$ by $G$ is defined to be 
\begin{equation}
X/\!\!/G := \Spec(\mathbb{C}[X]^G), 
\end{equation}
while a GIT quotient (twisted by $\chi$) is defined as 
\begin{equation} 
X/\!\!/_{\chi}G := \Proj\left(\bigoplus_{i\geq 0} \mathbb{C}[X]^{G,\chi^i}\right). 
\end{equation} 

Thus, invariant and semi-invariant functions play fundamental roles in the geometric construction of quotient spaces in algebraic geometry. 

\section{The singular locus of the components of $\mu_B$}
\label{section:singular-locus}

We investigated the singular locus using \texttt{Singular} (cf.~\cite{DGPS}) for $n\leq 5$.

\begin{theorem}
\label{thm:singular-locus}
Each irreducible component $\overline{\mathcal{C}}_{\nu}$, for $1\leq \nu\leq 2^n$, is singular, whose singular locus has codimension $1$ in $\overline{\mathcal{C}}_{\nu}$. 
Thus, the singular locus $\mu_B^{-1}(0)^{\sing}$ of the components of $\mu_B$ has codimension $1$ in $\mu_B^{-1}(0)$. 
It follows that the irreducible components do not intersect transversely. 
\end{theorem}

\begin{remark}
Recall from \cite[Thm.~1.2]{Im-Scrimshaw-parabolic} that the irreducible components are equidimensional of dimension $\binom{n+1}{2}+2n$. So the singular locus in each irreducible component has dimension $\binom{n+1}{2}+2n-1$. We note that we do not have an isolated singularity but since the singular locus is of high dimension, we cannot yet explicitly report a meaningful structure about this locus. 
\end{remark}

  
 
\begin{remark}
To describe the singular locus explicitly, we found the Jacobian of the polynomials $\mu_B(r,s,i,j)_{\gamma\nu}$ for $n\geq \gamma\geq \nu\geq 1$ and then set $n(n+1)/2\times n(n+1)/2$ minors equal to zero to describe the singular locus (cf. \cite{DGPS, GS,Mathematica}). 
Since this locus is difficult to generalize for a general $n$, we omit the details.  
\end{remark}

\section{Affine quotient of the Borel moment map}
\label{section:affine-B}
Recall Definition~\ref{definition:ell-J}. 
Let $\iota \in [n]$ and let $J\subseteq [n]$. 
Let 
\begin{equation}
\label{eqn:B-inv-polynomials}
\begin{split}
f_J(r,s,i,j) &= \tr(j \ell^J i), \\ 
g_J(r,s,i,j) &= \tr(\ell^J s), \\ 
h_{\iota}(r,s,i,j) &= \tr(L^{[n]\setminus \{\iota\} } r ), \\ 
k_{n,J}(r,s,i,j) &=\tr(j \ell^{[n]\setminus \{n\}}s \ell^{J} i),  \\ 
l_{J,1}(r,s,i,j) &=\tr(j \ell^{J}s \ell^{[n]\setminus\{1\}}i  ). \\ 
\end{split}
\end{equation}

\begin{remark} 
One can easily see that $k_{n,[n]\setminus \{ 1\}} = l_{[n]\setminus\{n\},1}$. 
\end{remark} 

\begin{remark}
Note that $k_{n,J}$ and $l_{J,1}$ are carefully chosen so that they are well-defined polynomials. That is, since $s\in \mathfrak{g}/\mathfrak{u}^+$, a representative of an equivalence class, we need to be careful that polynomials involving $s$ do not depend on the coordinates in $\mathfrak{u}^+$. 
\end{remark}

We have given in \cite[\S6.2]{MR3836769} rational versions of the polynomials in \eqref{eqn:B-inv-polynomials}, but in this paper, we have more $B$-invariant polynomials than the ones in \cite{MR3836769}.

\begin{lemma}
\label{lemma:B-inv-polys}
The polynomials in \eqref{eqn:B-inv-polynomials} are $B$-invariant. 
\end{lemma}

\begin{proof}
First, $f_J$ is a $B$-invariant polynomial since for $b\in B$,  
\begin{align*}
f_J(b.(r,s,i,j))  
 &= f_J(brb^{-1}, bsb^{-1}, bi, jb^{-1})  \\ 
 &= \tr(jb^{-1}\ell^J(\Ad_b (r)) bi )  \\ 
 &= \tr(jb^{-1}(b\ell^J b^{-1})bi ) \mbox{ by Lemma~\ref{lem:ellJ-B-invariant}}\\ 
 &= \tr(j\ell^Ji)
 = f_J(r,s,i,j).  
\end{align*} 
Similarly, $g_J$ is $B$-invariant since for $b\in B$, 
\begin{align*} 
g_J(b.(r,s,i,j)) 
 &= \tr(\ell^J(\Ad_b(r))bsb^{-1}) \\  
 &= \tr(\Ad_b(\ell^J(r))bsb^{-1}) \\ 
 &= \tr((b\ell^J(r)b^{-1})bsb^{-1}) \mbox{ by Lemma~\ref{lem:ellJ-B-invariant}}\\   
 &= \tr(\ell^J(r) s) = g_J(r,s,i,j). 
\end{align*} 
Next, for $b\in B$, we have 
\begin{align*}
h_{\iota}(b.(r,s,i,j)) 
&= \tr(L^{[n]\setminus \{\iota\}}(\Ad_b (r)) brb^{-1})  \\ 
&= \tr(\Ad_b(L^{[n]\setminus\{\iota\}}(r))brb^{-1})  \mbox{ by Corollary~\ref{cor:ellJ-B-invariant}}  \\ 
&=  \tr((b L^{[n]\setminus\{\iota\}}(r)b^{-1})brb^{-1}) \\ 
&= \tr(L^{[n]\setminus \{\iota\}}(r)r)  
  = h_{\iota}(r,s,i,j). 
\end{align*}
Now, for $b\in B$, we have 
\begin{align*} 
k_{n,J}(b.(r,s,i,j)) 
&= \tr(jb^{-1} \ell^{[n]\setminus \{n\}}(\Ad_b(r))b sb^{-1} \ell^{J}(\Ad_b(r)) bi) \\ 
&= \tr(jb^{-1}\Ad_b (\ell^{[n]\setminus \{n\}}(r)) bsb^{-1} \Ad_b(\ell^{J}(r)) bi)  
\mbox{ by Lemma~\ref{lem:ellJ-B-invariant}} \\ 
&= \tr(jb^{-1}(b\ell^{[n]\setminus \{n\}}(r)b^{-1}) bsb^{-1} (b\ell^{J}(r)b^{-1}) bi) \\ 
&= \tr(j \ell^{[n]\setminus \{n\}}(r) s \ell^{J}(r) i)   
  = k_{n,J}(r,s,i,j).   
\end{align*} 
Since the proof for  $l_{J,1}$ is similar to the one for $k_{n,J}$, we omit the proof. 
\end{proof}

\begin{proposition}
\label{prop:B-inv-subalgebra}
The $B$-invariant ring $\mathbb{C}[T^*(\mathfrak{b}\times \mathbb{C}^n)]^B$ contains the subalgebra generated by the polynomials in \eqref{eqn:B-inv-polynomials}. 
\end{proposition}

\begin{proof}
The theorem holds by Lemma~\ref{lemma:B-inv-polys}. 
Thus, we have an inclusion 
\[ 
\mathbb{C}[f_J, g_J, h_{\iota}, k_{n,J}, l_{J,1}] \hookrightarrow \mathbb{C}[T^*(\mathfrak{b}\times \mathbb{C}^n)]^B 
\]  
of rings.    
%
\end{proof}

\begin{conjecture}
Let $\iota\in [n]$ and $J\subseteq [n]$. 
The polynomials $f_J, g_J, h_{\iota}, k_{n,J}, l_{J,1}$ in \eqref{eqn:B-inv-polynomials} 
generate the $B$-invariant subalgebra 
$\mathbb{C}[T^*(\mathfrak{b}\times \mathbb{C}^n)]^B$. 
\end{conjecture}

\begin{remark}
Although $\mu_B$ is a complete intersection for  $n\leq 5$, 
the reverse inclusion 
\[ 
\mathbb{C}[T^*(\mathfrak{b}\times \mathbb{C}^n)]^B \subseteq \mathbb{C}[f_J, g_J, h_{\iota}, k_{n,J}, l_{J,1}] 
\] 
is much more difficult to prove for the nonreductive group setting. We thus leave this as an open problem. 
\end{remark}


\begin{remark}
One of the reasons why $B$-invariant polynomials are important is  because they provide an alternative way to study the affine quotient $\mu^{-1}(0)/\!\!/B$ via its ring of functions, i.e.,  since 
$\mathbb{C}[\mu^{-1}(0)/\!\!/B] \cong 
\mathbb{C}[T^*(\mathfrak{b}\times \mathbb{C}^n)]/\langle ([r,s]+ij)_{\gamma\nu} \rangle$,  
\begin{align*}
\mathbb{C}[\mu^{-1}(0)/\!\!/B] &= \mathbb{C}[\mu^{-1}(0)]^B  
= \frac{\mathbb{C}[T^*(\mathfrak{b}\times \mathbb{C}^n)]^B}{\langle ([r,s]+ij)_{\gamma\nu} \rangle^B}   
\supseteq \frac{\mathbb{C}[f_{J}, g_{J}, h_{\iota}, k_{n,J},l_{J,1}]}{
\langle \text{syzygies}\rangle + \langle ([r,s]+ij)_{\gamma\nu}\rangle^B}, 
\end{align*}
where the syzygies can explicitly be computed using \texttt{Macaulay2} (cf.~\cite{GS}). The syzygies, i.e., relations among generators, for our setting are inhomogeneous and appear to be difficult to generalize for higher $n$.
\end{remark}

   
%

\section{GIT quotients of the Borel moment map} 
\label{section:GIT-B}
Let $\chi:B\rightarrow \mathbb{C}^*$ be a character, i.e., a group homomorphism. In this section, we will consider when $\chi=\det$ and $\det^{-1}$. We refer to \cite[\S8]{GG06} for the construction of the semi-invariants for the classical setting $GL_n(\mathbb{C})$.

Consider the closed imbedding 
\begin{equation}
\label{eqn:diagonal-imbedding}
\begin{split}
\varepsilon_{i_0}:&\mathfrak{h}\times \mathfrak{h}\hookrightarrow \mathfrak{b}\times \mathfrak{b}^*\times \mathbb{C}^n \times (\mathbb{C}^n)^*  \\ 
&(a_1,\ldots, a_n, b_1,\ldots, b_n) \mapsto (\diag(a_1,\ldots, a_n), \diag(b_1,\ldots, b_n), i_0,0) 
\end{split}
\end{equation}
and 
\begin{equation}
\label{eqn:diagonal-imbedding-jvec}
\begin{split}
\varepsilon_{j_0}:&\mathfrak{h}\times \mathfrak{h}\hookrightarrow \mathfrak{b}\times \mathfrak{b}^*\times \mathbb{C}^n \times (\mathbb{C}^n)^*  \\ 
&(a_1,\ldots, a_n, b_1,\ldots, b_n) \mapsto (\diag(a_1,\ldots, a_n), \diag(b_1,\ldots, b_n), 0,j_0), 
\end{split}
\end{equation}
where $i_0=(1,\ldots,1)\in \mathbb{C}^n$ and $j_0=(1,\ldots, 1)\in (\mathbb{C}^n)^*$.

\subsection{Twisted by $\det$}
\label{subsection:twisted-det}
Let $J\subseteq [n]$, and 
let $v^*\in \wedge^n (\mathbb{C}^n)^*$ be a nonzero volume form. 
Let 
\begin{equation}
\label{eqn:defn-of-noncomm-ring-A}
f=(f_1,\ldots, f_n), 
\quad 
\mbox{ where } f_i\in A:= \mathbb{C}\left\langle 
\ell^J,   
\:\: L^{[n]\setminus \{\iota\}} r, \:\: \ell^{[n]\setminus \{n\}}s \ell^{J}, \:\: \ell^{J}s \ell^{[n]\setminus\{1\}}\right\rangle 
\end{equation}
and $f$ is an $n$-tuple of noncommutative polynomials. 
Consider polynomial functions of the form 
\begin{equation}
\label{eqn:det-semi-invariant}
\psi_f = \left\langle v^*, f_1 i\wedge \ldots \wedge f_n i \right\rangle \in \mathbb{C}[T^*(\mathfrak{b}\times \mathbb{C}^n)].   
\end{equation}

\begin{lemma}
\label{lemma:det-semi-inv}
The polynomials $\psi_f$ are $\det$-semi-invariant. 
\end{lemma}

\begin{proof}
Consider $\psi_f$ in \eqref{eqn:det-semi-invariant}. 
First, it is clear that $\ell^J$ and $L^{[n]\setminus \{\iota\}}r$ are well-defined since they are elements in $\mathfrak{b}$. 
Next, we will check that the product of matrices 
\[ 
\ell^{[n]\setminus \{n\}}s \ell^{J} 
\quad 
\mbox{ and }
\quad
\ell^{J}s \ell^{[n]\setminus\{1\}}
\] 
are well-defined. 
First, consider $\ell^{[n]\setminus \{n\}}s \ell^{J}$.
Since 
\[ 
\ell^{[n]\setminus \{n\}}_{\gamma\mu} = 
\begin{cases} 
* &\mbox{ if }\mu=n, \\ 
0 &\mbox{ otherwise},  \\ 
\end{cases} 
\] 
where $*$ represents a nonzero coordinate entry, 
$\ell^{[n]\setminus \{n\}}$ in $\ell^{[n]\setminus \{n\}}s$ will kill all the coordinates of $s$ in $\mathfrak{u}^+$, i.e., $\ell^{[n]\setminus \{n\}}s$ does not depend on $\mathfrak{u}^+$. Thus,  
$\ell^{[n]\setminus \{n\}}s \ell^{J}$ is well-defined. 

Similarly, for $\ell^{J}s \ell^{[n]\setminus\{1\}}$, 
\[ 
\ell^{[n]\setminus\{1\}}_{\gamma\mu}
= \begin{cases}
* &\mbox{ if } \gamma = 1,  \\ 
0 &\mbox{ otherwise}, 
\end{cases} 
\] 
where $*$ represents a nonzero coordinate entry. 
So $s \ell^{[n]\setminus\{1\}}$ only involves the first column of $s$. Thus, $s \ell^{[n]\setminus\{1\}}$ doesn't depend on $\mathfrak{u}^+$, and $\ell^{J}s \ell^{[n]\setminus\{1\}}$ is well-defined. 

Next, since we have 
\begin{align*}
\ell^J(\Ad_b(r)) &= \Ad_b(\ell^J(r)) \mbox{ by Lemma~\ref{lem:ellJ-B-invariant}}, \\ 
L^{[n]\setminus \{\iota\}}(\Ad_b(r)) \Ad_b(r) 
&= \Ad_b(L^{[n]\setminus \{\iota\}}(r)) \Ad_b(r) 
=  \Ad_b(L^{[n]\setminus \{\iota\}}(r) r) \mbox{ by Cor.~\ref{cor:ellJ-B-invariant}}, \\ 
\ell^{[n]\setminus \{n\}}(\Ad_b(r))\Ad_b(s) \ell^{J} (\Ad_b(r))
&= \Ad_b(\ell^{[n]\setminus \{n\}}(r) s \ell^{J}(r)) \mbox{ by Lemma~\ref{lem:ellJ-B-invariant}}, \mbox{ and } \\ 
\ell^{J}(\Ad_b(r)) \Ad_b(s) \ell^{[n]\setminus\{1\}}(\Ad_b(r))  
&= \Ad_b(\ell^{J}(r) s \ell^{[n]\setminus\{1\}}(r)) 
\mbox{ by Lemma~\ref{lem:ellJ-B-invariant}},  
\end{align*}
$\Ad_b(f_{\nu}(r,s))= f_{\nu}(\Ad_b(r),\Ad_b(s))$ for $1\leq \nu\leq n$. 
So 
\[ 
\Ad_b(f_{\nu}i) = \Ad_b(f_{\nu})b i = bf_{\nu}b^{-1}bi = bf_{\nu}i 
\quad 
\mbox{ for }
1\leq \nu\leq n. 
\]  
Thus 
\[ 
b.\psi_f 
= \langle v^*, bf_1i\wedge \ldots \wedge bf_n i \rangle
= \det(b) \langle v^*, f_1i\wedge \ldots \wedge f_n i \rangle
= \det(b)\psi_f, 
\] 
and this completes the proof. 
\end{proof}

\begin{theorem}
\label{thm:det-semi-inv}
For each $k\geq 1$, restriction of functions via the imbedding $\varepsilon_{i_0}$ in \eqref{eqn:diagonal-imbedding} induces a vector space surjection   
$\varepsilon_{i_0}^*:\mathbb{C}[\mu_B^{-1}(0)]^{B,\det^k} \twoheadrightarrow D^k$, where 
\begin{equation}
\label{eqn:alternating-submodule}
D:= \mathbb{C}[\ell^J,  
\: L^{[n]\setminus \{\iota\}} r, \: \ell^{[n]\setminus \{n\}}s \ell^{J}, \: \ell^{J}s \ell^{[n]\setminus\{1\}}]^{\epsilon}, 
\end{equation}
a subspace of $S_n$-alternating polynomials. 
\end{theorem}

\begin{proof}
Let $k\geq 1$. The imbedding $\overline{\mathcal{C}}_0\hookrightarrow \mu_B^{-1}(0)$ induces a bijection $\mathbb{C}[\mu_B^{-1}(0)]^{B,\det^k}\stackrel{\sim}{\rightarrow}\mathbb{C}[\overline{\mathcal{C}}_0]^{B,\det^k}$. It follows that $\mathbb{C}[\mu_B^{-1}(0)]^{B,\det^k}$  contains products $\psi_1\cdots \psi_k$ as a $\mathbb{C}[\mu_B^{-1}(0)]^B$-module.
Since $\varepsilon_{i_0}^*\psi_{f}\in D$ for any $f$, we see that $\varepsilon_{i_0}^*(\psi_1\cdots \psi_k)\in D^k$. 
Hence 
\[ 
\varepsilon_{i_0}^*(\mathbb{C}[\mu_B^{-1}(0)]^{B,\det^k})=\varepsilon_{i_0}^*(\mathbb{C}[\overline{\mathcal{C}}_0]^{B,\det^k})\subseteq D^k. 
\]  


Note that $\mathbb{C}[\mu_B^{-1}(0)]^{B,\det^k}$ contains the $k$-fold products $\psi_1\cdots \psi_k$, as a $\mathbb{C}[\mu_B^{-1}(0)]^{B}$-module.
Thus it suffices to prove surjectivity of the map $\varepsilon_{i_0}^*$ for $k=1$. 
To prove this, we identify $D$ with the $n$-th exterior power  $\wedge^n E$ of the vector space of polynomials, where 
\[ 
E := \mathbb{C}[\ell^J,   
\: L^{[n]\setminus \{\iota\}} r, \: \ell^{[n]\setminus \{n\}}s \ell^{J}, \: \ell^{J}s \ell^{[n]\setminus\{1\}}]. 
\]  
With this identification, the space $D$ is spanned by wedge products $f_1\wedge \ldots \wedge f_n$,  where  
$f_1,\ldots, f_n\in  E$. 

By the definition of the irreducible component $\overline{\mathcal{C}}_0$ for any $(r,s,i,j)\in \overline{\mathcal{C}}_0$, we have $[r,s]=[r,s]+ij=0$. So for any $f\in E$, the expression $f$ is a well-defined matrix. That is, for any lift of $f$ to a noncommutative polynomial 
$\widehat{f}\in A$ (cf.~\eqref{eqn:defn-of-noncomm-ring-A}), i.e., for any $\widehat{f}$ in the preimage of $f$ under the natural projection $A \twoheadrightarrow E$, we have $\widehat{f}=f$. Thus, given an $n$-tuple $f_1,\ldots, f_n$, we have a well-defined element 
\[ 
\psi_f = \langle v^*, f_1i\wedge \ldots \wedge f_n i\rangle \in \mathbb{C}[\mu_B^{-1}(0)]^{B,\det}. 
\] 
\end{proof}

%

\subsection{Twisted by $\det^{-1}$}
\label{subsection:twisted-det-inverse}
Let $v\in \wedge^n \mathbb{C}^n$ be a nonzero volume form. Similar as before, let 
\[
g = (g_1,\ldots, g_n), \qquad 
\mbox{ where }
g_i \in  A,   
\] 
$A$ is defined in \eqref{eqn:defn-of-noncomm-ring-A}, 
and $g$ is an $n$-tuple of noncommutative polynomials. Consider polynomials of the form 
\begin{equation}
\label{eqn:det-inv-semi-inv}
\phi_g = \langle  jg_1\wedge \ldots \wedge jg_n, v \rangle \in \mathbb{C}[T^*(\mathfrak{b}\times \mathbb{C}^n)].
\end{equation}

\begin{lemma}
\label{lemma:det-inv-semi-inv}
The polynomials $\phi_g$ are $\det^{-1}$-semi-invariant. 
\end{lemma}

\begin{proof}
The proof of Lemma~\ref{lemma:det-inv-semi-inv}
is analogous to the proof of Lemma~\ref{lemma:det-semi-inv}, so we omit the details.  
\end{proof}

Similar to Theorem~\ref{thm:det-semi-inv}, we have the following: 
\begin{theorem}
\label{thm:det-inv-semi-inv}
For each $k\leq 1$, restriction of functions via the imbedding $\varepsilon_{j_0}$ in \eqref{eqn:diagonal-imbedding-jvec} induces a vector space surjection $\varepsilon_{j_0}^*:\mathbb{C}[\mu_B^{-1}(0)]^{B,\det^k} \twoheadrightarrow  
D^{-k}$, where $D$ is the $S_n$-alternating submodule given in \eqref{eqn:alternating-submodule}. 
\end{theorem}

%

\section{Restriction to the regular semi-simple and the singular locus}
\label{section:restriction-to-rss-singular-locus} 

When restricting both the GIT and the affine quotient to the regular semi-simple locus $\mu_B^{-1}(0)^{\rss}$, we expect an isomorphism 
\[ 
\xymatrix@-1pc{
\mu_B^{-1}(0)^{\rss}/\!\!/_{\det}B \cong 
\mu_B^{-1}(0)^{\rss}/\!\!/_{\det^{-1}}B  \ar[rr]^{\:\:\:\qquad\qquad\simeq} & &  \mu_B^{-1}(0)^{\rss}/\!\!/B
}
\] 
of varieties. 
In fact, the regular semi-simple locus should precisely be the nonsingular locus in both the GIT and affine quotient.  By \cite[Thm~1.6]{MR3836769}, the singular locus $\mu_B^{-1}(0)^{\sing}/\!\!/B$ in the $B$-affine quotient is contained in the locus isomorphic to 
\[ 
H=\{(r_1,\ldots, r_n,0,\ldots, 0):r_i = r_j \mbox{ for some }i\not=j\}, 
\]  
and 
\[ 
\xymatrix@-1pc{
\mu_B^{-1}(0)^{\sing}/\!\!/_{\det}B \stackrel{\dagger}{\cong}
\mu_B^{-1}(0)^{\sing}/\!\!/_{\det^{-1}}B \ar@{->>}[rr] & & 
\mu_B^{-1}(0)^{\sing}/\!\!/ B 
}
\] 
is expected to be a resolution of singularities, 
where $\dagger$ is meant that a certain notion of wall-crossing in GIT has occurred, and a notion of stability has been changed from $\det$ to $\det^{-1}$. 
In fact, we expect $\mu_B^{-1}(0)^{\sing}/\!\!/B \cong H$. 
Note that singularities are still not isolated in the affine quotient.

\section{Connections to the Hilbert scheme of $n$ points on a complex plane} 
\label{section:connections-to-Hilbert-scheme} 
Let $\mu_G:T^*(\mathfrak{gl}_n\times \mathbb{C}^n)\rightarrow \mathfrak{gl}_n^*\stackrel{\tr}{\cong}\mathfrak{gl}_n$, where $(r,s,i,j)\mapsto [r,s]+ij$.  
By \cite[Ch.~3]{MR1711344}, we have the well-known morphism 
\[
\xymatrix@-1pc{ 
\mu_G^{-1}(0)/\!\!/_{\det} GL_n(\mathbb{C}) \cong (\mathbb{C}^2)^{[n]}\ar@{->>}[dd]^{\text{Hilbert--Chow}} \\ 
\\  
\mu_G^{-1}(0)/\!\!/_{\det} GL_n(\mathbb{C}) \cong S^n\mathbb{C}^2 = \mathbb{C}^2\times \ldots \times \mathbb{C}^2/S_n, \\ 
}
\] 
where $(\mathbb{C}^2)^{[n]}$ is the Hilbert scheme of $n$ points on a complex plane. 

Analogously for our setting, we have: 
\begin{conjecture}
\label{conjecture:relating-to-Hilbert-schemes}
The following hold: 
\begin{enumerate}
\item\label{item:variational-GIT} $\mu_B^{-1}(0)/\!\!/_{\det}B \cong \mu_B^{-1}(0)/\!\!/_{\det^{-1}}B$, 
\item there is a resolution $\mu_B^{-1}(0)/\!\!/_{\det}B\twoheadrightarrow\mu_B^{-1}(0)/\!\!/B$ of singularities, 
\item $\mu_B^{-1}(0)/\!\!/_{\det}B$ is isomorphic to the flag Hilbert scheme on a complex plane (line). 
\end{enumerate}
\end{conjecture}
Conjecture~\ref{conjecture:relating-to-Hilbert-schemes}\eqref{item:variational-GIT} is known as variational GIT, or wall-crossing. 

\begin{conjecture}
\label{conjecture:commuting-morphisms}
The following diagram 
\[
\xymatrix@-1pc{
\ar@{->>}[dd]
\mu_B^{-1}(0)/\!\!/_{\det}B \ar@{..>}[rr] & & \mu_G^{-1}(0)/\!\!/_{\det} GL_n(\mathbb{C}) \ar@{->>}[dd] \\ 
& \CircleArrowleft & \\ 
\mu_B^{-1}(0)/\!\!/B \ar@{..>}[rr] & & \mu_G^{-1}(0)/\!\!/GL_n(\mathbb{C})  \\ 
}
\] 
commutes. 
\end{conjecture}
 
\begin{remark}
Although proving Conjectures~\ref{conjecture:relating-to-Hilbert-schemes} and \ref{conjecture:commuting-morphisms} would be interesting, we have obtained the syzygies for our $B$-invariant and $\det$-semi-invariant polynomials using \cite{GS}. Thus, although writing these syzygies in terms of our generators appear to be unsystematic and difficult to generalize for higher $n$, constructing affine and GIT quotients with our current set of generators and relating them to other well-known schemes would still be facinating and interesting. 
\end{remark}


\bibliographystyle{amsalpha} 
\bibliography{GS-GIT-affine}   

\def\cprime{$'$} \def\cprime{$'$}
\providecommand{\bysame}{\leavevmode\hbox to3em{\hrulefill}\thinspace}
\providecommand{\MR}{\relax\ifhmode\unskip\space\fi MR }
\providecommand{\MRhref}[2]{%
  \href{http://www.ams.org/mathscinet-getitem?mr=#1}{#2}
}
\providecommand{\href}[2]{#2}
\begin{thebibliography}{DGPS19}

\bibitem[CG10]{MR2838836}
Neil Chriss and Victor Ginzburg, \emph{Representation theory and complex
  geometry}, Modern Birkh\"{a}user Classics, Birkh\"{a}user Boston, Inc.,
  Boston, MA, 2010, Reprint of the 1997 edition.

\bibitem[DGPS19]{DGPS}
Wolfram Decker, Gert-Martin Greuel, Gerhard Pfister, and Hans Sch\"onemann,
  \emph{{\textsc Singular} {4-1-2} --- {A} computer algebra system for
  polynomial computations}, Available at
  \href{http://www.singular.uni-kl.de}{http://www.singular.uni-kl.de}, 2019.

\bibitem[GG06]{GG06}
Wee~Liang Gan and Victor Ginzburg, \emph{Almost-commuting variety, {$\mathscr
  D$}-modules, and {C}herednik algebras}, IMRP Int. Math. Res. Pap. (2006),
  26439, 1--54, With an appendix by Ginzburg.

\bibitem[GNR16]{GNR16}
Eugene Gorsky, Andrei Negu{\c{t}}, and Jacob Rasmussen, \emph{Flag {H}ilbert
  schemes, colored projectors and {K}hovanov-{R}ozansky homology}, arXiv
  preprint arXiv:1608.07308 (2016).

\bibitem[GS02]{GS}
Daniel~R. Grayson and Michael~E. Stillman, \emph{Macaulay2, a software system
  for research in algebraic geometry}, Available at
  \href{https://faculty.math.illinois.edu/Macaulay2}{https://faculty.math.illinois.edu/Macaulay2},
  1993-2002.

\bibitem[Hai01]{MR1839919}
Mark Haiman, \emph{Hilbert schemes, polygraphs and the {M}acdonald positivity
  conjecture}, J. Amer. Math. Soc. \textbf{14} (2001), no.~4, 941--1006.

\bibitem[Im14]{MR3312842}
Mee~Seong Im, \emph{On semi-invariants of filtered representations of quivers
  and the cotangent bundle of the enhanced {G}rothendieck-{S}pringer
  resolution}, ProQuest LLC, Ann Arbor, MI, 2014, Thesis (Ph.D.)--University of
  Illinois at Urbana-Champaign.

\bibitem[Im18]{MR3836769}
\bysame, \emph{The regular semisimple locus of the affine quotient of the
  cotangent bundle of the {G}rothendieck-{S}pringer resolution}, J. Geom. Phys.
  \textbf{132} (2018), 84--98.

\bibitem[IS]{Im-Scrimshaw-parabolic}
Mee~Seong Im and Travis Scrimshaw, \emph{The regularity of almost-commuting
  partial {G}rothendieck--{S}pringer resolutions and parabolic analogs of
  {C}alogero--{M}oser varieties}, arXiv:1812.02283.

\bibitem[KL09]{MR2525917}
Mikhail Khovanov and Aaron~D. Lauda, \emph{A diagrammatic approach to
  categorification of quantum groups. {I}}, Represent. Theory \textbf{13}
  (2009), 309--347.

\bibitem[KL11]{MR2763732}
\bysame, \emph{A diagrammatic approach to categorification of quantum groups
  {II}}, Trans. Amer. Math. Soc. \textbf{363} (2011), no.~5, 2685--2700.

\bibitem[MFK94]{MR1304906}
D.~Mumford, J.~Fogarty, and F.~Kirwan, \emph{Geometric invariant theory}, third
  ed., Ergebnisse der Mathematik und ihrer Grenzgebiete (2) [Results in
  Mathematics and Related Areas (2)], vol.~34, Springer-Verlag, Berlin, 1994.

\bibitem[Nak99]{MR1711344}
Hiraku Nakajima, \emph{Lectures on {H}ilbert schemes of points on surfaces},
  University Lecture Series, vol.~18, American Mathematical Society,
  Providence, RI, 1999.

\bibitem[Nev11]{Nevins-GSresolutions}
Thomas Nevins, \emph{Stability and {H}amiltonian reduction for
  {G}rothendieck-{S}pringer resolutions}, Available at
  \href{http://www.math.illinois.edu/~nevins/papers/b-hamiltonian-reduction-2011-0316.pdf}%
  {http://www.math.illinois.edu/\urltilde
  nevins/papers/b-hamiltonian-reduction-2011-0316.pdf}, 2011.

\bibitem[New09]{MR2537067}
P.~E. Newstead, \emph{Geometric invariant theory}, Moduli spaces and vector
  bundles, London Math. Soc. Lecture Note Ser., vol. 359, Cambridge Univ.
  Press, Cambridge, 2009, pp.~99--127.

\bibitem[Res]{Mathematica}
Wolfram Research, \emph{Mathematica, {V}ersion 11.3}, Champaign, IL, 2018.

\bibitem[Rou08]{rouquier2008}
Raphael Rouquier, \emph{2-{K}ac-{M}oody algebras}, Available at
  \href{http://arxiv.org/abs/0812.5023}{http://arxiv.org/abs/0812.5023}, 2008.

\bibitem[Rou12]{MR2908731}
Rapha{\"e}l Rouquier, \emph{Quiver {H}ecke algebras and 2-{L}ie algebras},
  Algebra Colloq. \textbf{19} (2012), no.~2, 359--410.

\end{thebibliography}

\end{document}